\documentclass[a4paper,12pt,intlimits,oneside]{amsart}

\usepackage{enumerate}
\usepackage{amsfonts,amsmath}

\usepackage{latexsym,amssymb}

\usepackage{mathrsfs}

\newcommand{\comment}[1]{}

\newcommand{\bC}{{\mathbb C}}

\newcommand{\bR}{{\mathbb R}}

\newcommand{\R}{{\mathbb R}}

\def\H{{\mathcal H}}

\def\la{\langle}
\def\ra{\rangle}

\def\BMO{B\! M\! O}

\def\BMOA{B\! M\! O\! A}

\newcounter{rea}
\begin{document}

\title[]{Factorization of some Hardy type spaces of holomorphic functions}      
\author{\bf{Aline Bonami and Luong Dang Ky}}
\address{MAPMO-UMR 6628,
D\'epartement de Math\'ematiques, Universit\'e d'Orleans, 45067
Orl\'eans Cedex 2, France} 
\email{{\tt
aline.bonami@univ-orleans.fr}}
\address{Department of Mathematics, University of Quy Nhon, 170 An Duong Vuong, Quy Nhon, Binh Dinh, Viet Nam} 
\email{{\tt dangky@math.cnrs.fr}}

\keywords{BMOA, holomorphic functions, Hardy spaces}
\subjclass[2010]{47B35(32A35, 42B35)}

\begin{abstract}
We prove that the pointwise product of two holomorphic functions of the upper half-plane, one in the Hardy space $\mathcal H^1$, the other one in its dual, belongs to a Hardy type space. Conversely, every holomorphic function in this space can be written as such a product. This generalizes previous characterization in the context of the unit disc.
\end{abstract}

\maketitle
\newtheorem{theorem}{Theorem}[section]
\newtheorem{lemma}{Lemma}[section]
\newtheorem{proposition}{Proposition}[section]
\newtheorem{remark}{Remark}[section]
\newtheorem{corollary}{Corollary}[section]
\newtheorem{definition}{Definition}[section]
\newtheorem{example}{Example}[section]
\numberwithin{equation}{section}
\newtheorem{Theorem}{Theorem}[section]
\newtheorem{Lemma}{Lemma}[section]
\newtheorem{Proposition}{Proposition}[section]
\newtheorem{Remark}{Remark}[section]
\newtheorem{Corollary}{Corollary}[section]
\newtheorem{Definition}{Definition}[section]
\newtheorem{Example}{Example}[section]
\newtheorem*{theoremjj}{Theorem J-J}
\newtheorem*{theorema}{Theorem A}
\newtheorem*{theoremb}{Theorem B}
\newtheorem*{theoremc}{Theorem C}

\section{Introduction }

Let $\bC_+$ be the upper half-plane in the complex plane. We recall that, for $p>0$,  the holomorphic Hardy space $\mathcal H^p_a(\bC_+)$ is defined as the space of holomorphic functions  $f$ such that
\begin{equation}
 \|f\|_{\mathcal{H}^p_a}^p:=\sup_{y>0}\int_{-\infty}^{+\infty}|f(x+iy)|^pdx <\infty.
 \end{equation}
 By Fefferman's Theorem, the dual space of $\mathcal H^1_a(\bC_+)$ is the space $\BMOA(\bC_+)$. Here we are not interested by the definition of the Hermitian scalar product $\la f, g\ra$ when $f$ is in $\mathcal H^1_a(\bC_+)$ and $g$ is in $\BMOA(\bC_+)$, but by the pointwise product $fg$ of the two holomorphic functions. We identify the space of such products. This has already been done in the case of the unit disc in \cite{BIJZ}, where one finds a Hardy-Orlicz space. The novelty here is the fact that one has also to take into account the behavior at infinity. The space under consideration belongs to the family of Hardy spaces of Musielak type. It has been introduced in the setting of real Hardy spaces in \cite{BGK}.

Before stating the theorem, let us recall that $g$ is in $BMOA(\bC_+)$ if and only if one of the two equivalent conditions are satisfied.{\sl
\begin{enumerate}[(i)]
\item $|g'(x+iy)|^2 y \, dx\, dy$ is a Carleson measure.
\item $g$ can be written as
\begin{equation}\label{Bonami 1}
g(x+ iy) = P_y*g_0(x),
\end{equation}
  where $g_0$ belongs to $BMO(\mathbb R)$ and its Fourier transform is supported in $ [0,\infty)$.
\end{enumerate}} Here $P_y$ is the Poisson kernel.
Next  we define  $L^{\log}(\bR)$, as in \cite{BGK}, as the space of measurable functions $f$ such that
$$\int_{\mathbb R} \frac {|f(x)|}{ \log(e+|x|) +\log (e+ |f(x)|)}dx <\infty.$$
This is a particular case  of a Musielak-Orlicz space, defined as  the space of measurable functions $f$ such that
$$\int_{\bR} \theta(x, |f(x)|) dx <\infty$$
under adequate assumptions on $\theta$ (see \cite{Ky} for details). For $f\in L^{\log}(\bR)$, we define the "norm" by
$$\|f\|_{L^{\log}}:=\inf \left\{\lambda>0 : \int_{\bR} \theta(x, |f(x)|/\lambda) dx \leq 1\right\},$$
with $\theta(x, t)=t(1+\log_+(|x|) +\frac 12\log_+ ( t))^{-1}$. The choice of this particular function, whose ratio with $ \log(e+|x|) +\log (e+ |f(x)|)$ is bounded above and below, guarantees that $t\mapsto \theta(x, t^2)$ is a convex function, which will be useful later.

We then define $\mathcal H^{\log}_a(\bC_+)$ as the space of holomorphic functions $f$ such that
$$\|f\|_{\mathcal H^{\log}_a}= \sup_{y>0} \|f(\cdot + iy)\|_{L^{\log}}<\infty.$$

Our main result is the following one.

\begin{theorem}\label{the main theorem}
The product of $f\in \mathcal H^1_a(\bC_+)$ and $g\in BMOA(\bC_+)$ belongs to $\mathcal H^{\log}_a(\bC_+)$. Moreover, every function in  $\mathcal H^{\log}_a(\bC_+)$ can be written as such a product. In other words,
$$\mathcal H^1_a(\bC_+) .  BMOA(\bC_+) = \mathcal H^{\log}_a(\bC_+).$$
\end{theorem}

As a consequence, by using standard methods, we can identify the class of holomorphic symbols $b$ of Hankel operators $H_b$ that extend into continuous antilinear operators on $\mathcal H_a^1(\bC_+)$. For simplicity we only consider symbols $b$ that are bounded and define
$$ H_b(f)=P(b\overline f)$$
with $P$ the Cauchy operator, which extends the orthogonal projection from $L^2(\mathbb{R})$ onto the subspace of functions whose Fourier transforms are supported in $[0, \infty)$. Here and what in follows, $I(x_0,r)$ stand for the interval $(x_0-r,x_0+r)$ in $\mathbb R$.

\begin{theorem}\label{hankel}
For $b$ a bounded holomorphic function in $\bC_+$, the operator $H_b$ extends into a bounded antilinear operator on $\mathcal H_a^1(\bC_+)$ if and only if $b$ belong to the space $\BMOA^{\log}(\bC_+)$, that is,
\begin{equation}
\sup_{I(x_0, r)}\frac{|\log r|+\log (e+|x_0|) }{r}\int_{T(I(x_0, r)}|\nabla b(x+iy)|^2ydx dy<\infty,
\end{equation}
where the supremum is taken over all intervals $I(x_0,r)\subset \mathbb R$ and $T(I(x_0,r))$ denotes the tent on $I(x_0,r)$.
\end{theorem}


\section{The space $\mathcal H^{\log}_a(\bC_+)$}

In this section, we extend  to the space $\mathcal H^{\log}_a(\bC_+)$ those  properties of Hardy spaces that  we will need. Let us first define, more generally, spaces $L^{\varphi}(\R)$ and $\mathcal H^{\varphi}_a(\bC_+)$, with the specific function $\theta$ replaced by $\varphi$. We will only use two other specific functions, namely
\begin{equation}\label{functions}
\theta_0(x, t)= \theta(x, t^2)\qquad \qquad \theta_1(x, t) = \theta(x,t)^2.
\end{equation} Both are convex functions, of upper and lower type $2$. So $\mathcal H^{\theta_0}_a(\bC_+)$ and $\mathcal H^{\theta_1}_a(\bC_+)$ are Banach spaces, while $\mathcal H^{\log}_a(\bC_+)$ is not a normed space.

As usual, the operator $M$ will be the classical Hardy-Littlewood maximal operator.
The nontangential maximal function of  a  function $f$ defined  in $\bC_+$ is given by
\begin{equation}
f^*(x)= \sup_{z\in \Gamma(x)} |f(z)|,
\end{equation}
where $\Gamma(x)= \{z= u +iy\in\bC_+: |u-x|<y\}$.

The next theorem gives the characterization of the space $\mathcal H^{\log}_a(\bC_+)$.

\begin{theorem}\label{maximal characterization}
$f\in \mathcal H^{\log}_a(\bC_+)$ if and only if $f^* \in L^{\log}(\mathbb R)$. Moreover,
$$\|f\|_{\mathcal H^{\log}_a}\sim \|f^*\|_{L^{\log}}.$$
\end{theorem}
\begin{proof}
One implication is obvious. Let us prove the other one. We consider $f\in \mathcal H^{\log}_a(\bC_+)$. We use the fact that $g=|f|^{1/2}$ is a sub-harmonic function and satisfies the inequality
$$ \sup_{y>0} \|g(\cdot + iy)\|_{L^ {\theta_0}}=\|f\|_{\mathcal H^{\log}_a}<\infty.$$
It is easy to adapt to the present situation the classical theorems. Namely, norms $\|g(\cdot + iy)\|_{L^{ \theta_0}}$ are decreasing and there exists a boundary value $g_0\in L^ {\theta_0}(\R)$ such that $g(x+iy)\leq P_y*g_0(x)$. So,
\begin{equation}\label{a maximal estimate 3/6}
(f^*)^{1/2}(x)\leq \sup_{u+iy\in \Gamma(x)} P_y*g_0(u)\leq C M(g_0)(x).
\end{equation}
By the $L^{ \theta_0}$-boundedness of $M$ (see \cite[Corollary 2.8]{LHY}), we obtain that
$$ \|f^*\|_{L^{\log}} \leq C \|g_0\|_{L^{ \theta_0}}\leq C \|f\|_{\mathcal H^{\log}_a}$$
for some uniform constant related with the norm of $M$ in $L^{\theta_0}(\mathbb R)$.
 \end{proof}
 We will also need the next statement.
 \begin{theorem}\label{a theorem for boundary value for Hardy spaces}
  Let $f$ be a function in $\mathcal H^{\log}_a(\bC_+)$. Then $f$ has  nontangential limits almost everywhere. Moreover, if the boundary value function
 $$f_0(x):=\lim_{\Gamma(x) \ni z\to x} f(z)$$
belongs to the space $L^1(\mathbb R)$, then $f$ is in $\mathcal H^1_a(\bC_+)$.

\end{theorem}

\begin{proof}
The proof of the existence of a. e. nontangential limits is quite similar to the classical one (see Garnett's book \cite{Ga}, see also \cite{BGK2} for more details). It will be omitted. Assume that $f_0$ is in $L^1(\bR)$. Since $\mathcal H^1_a(\bC_+)$ is a subspace of $\mathcal H^{\log}_a(\bC_+)$, we can assume that $f_0=0$. We proceed as above and consider the subharmonic function $g=|f|^{1/2}$ whose boundary values are $0$. This forces $f$ to be $0$.
\end{proof}

\section{Factorization of the space $\mathcal H^{\log}_a(\bC_+)$}

Since we consider functions in $BMOA(\bC_+)$ and not only equivalence classes, we define a norm on this space. For a function $f\in BMO(\mathbb R)$, following \cite{BIJZ}, we define the norm by
$$\|f\|_{BMO^+}:= \|f\|_{BMO} + \int_{I(0,1)} |f(x)|dx.$$

Let $f$ be a function in $BMOA(\bC_+)$ given by (\ref{Bonami 1}),
with $f_0\in BMO(\mathbb R)$. We define
$$\|f\|_{BMOA^+}:= \|f_0\|_{BMO^+}.$$

 Here and in future, we denote by $m_B f$ the average value of $f$ over the ball $B$. Constants $C$ may vary from line to line.

 The next lemma gives a bound of norms in $BMO(\bR)$ on lines that are parallel to the $x$-axis.

 \begin{lemma}\label{Bonami 4/6}
 Let $f$ be a function in $BMOA(\bC_+)$. Then there exists some constant $C$ such that, for all $y>0$,
  \begin{equation}\label{BKsmall}
  \|f(\cdot + iy)\|_{BMO^+}\leq C \log(e+y) \|f\|_{BMOA^+}.
  \end{equation}
\end{lemma}

\begin{proof}
Let us first prove \eqref{BKsmall}. Since $\BMO(\bR)$ is invariant by translation, we already know that
$$\|f(\cdot+iy)\|_{BMO}= \|P_y*f_0\|_{BMO} \leq C \|f_0\|_{BMO}.$$
So it is sufficient to prove that
\begin{equation}\label{small}
\int_{I(0,1)} |f(x + iy)| dx = \int_{I(0,1)} |P_y * f_0(x)|dx \leq C \log(e+y) \|f_0\|_{BMO^+}.
\end{equation}
Remark that $y^2 + |x-u|^2 \sim a^2 + |u|^2$ for $ x\in I(0,1), u\notin I(0,a)$,  with $a:=2\max\{1, y\}$. We cut $f_0$ into $f_0\chi_{I(0, a)
)}+ f_0 \chi_{I(0, a)^c}$. Since the convolution by the Poisson kernel has norm $1$ in $L^1(\bR)$, the first $L^1$-norm is bounded by $m_{I(0,a)}(|f_0|)$. So
$$\int_{I(0,1)} |P_y * f_0(x)|dx\leq m_{I(0,a)}(|f_0|) +C a \int_{\mathbb R} \frac{|f_0(u)|}{a^2 + |u|^2} du.$$
We then use the standard inequalities, valid for all functions $g\in\BMO(\R)$,
$$a\int_{\mathbb R} \frac{|g(u)-m_{I(0, a)}g|}{a^2 + |u|^2} dt\leq C \|g\|_{BMO}$$
and
$$m_{I(0,a)}g\leq m_{I(0,1)}g+ C\log a \|g\|_{BMO}$$
to conclude that
$$\int_{I(0,1)} |P_y * f_0(x)|dx\leq C \log(e+y) \|f_0\|_{BMO^+}$$
since $a=2\max\{1,y\}\leq 2(e+y)$. This ends the proof.

\end{proof}

With this lemma we conclude for one side of the theorem.

\begin{proposition}\label{fill the gap}There exists  a constants $C$  such that for every $f\in \mathcal H^1_a(\bC_+)$ and $g\in BMOA(\bC_+)$, the pointwise product $fg$ is in $\mathcal H^{\log}_a(\bC_+)$ and satisfies
\begin{equation}\label{apriori}
\|fg\|_{\mathcal H^{\log}_a}\leq C \|f\|_{\mathcal H^1_a} \|g\|_{BMOA^+}.
\end{equation}
\end{proposition}

\begin{proof}
We start by an a priory estimate. From Corollary 3.3 in the book of Garnett \cite{Ga}, we know that continuous functions  $f$ on $\overline{\bC_+}$ such that
$$A(f):=\sup_{z\in\bC_+}(1+|z|)^3 |f(z)|<\infty$$
are dense. Using this last assumption and \eqref{BKsmall}, we find that
$$\int_{\bR} |f(x+iy)||g(x +iy)|dx\leq \frac{  A(f) }{1+y}\int_{\bR}\frac{|g(x+iy)|}{1+x^2}dx\leq CA(f)\|g\|_{BMOA^+}.$$
So the product $h:=fg$ is in $\mathcal H^1_a(\bC_+)$ and  $h(\cdot+iy)=P_y*h_0$, where $h_0$ is the boundary value of $h$. By the same argument as in the proof of Theorem \ref{maximal characterization}, and using the fact that $\theta_0(\cdot, t)$ is convex, we prove that $h(\cdot+iy)=P_y*h_0$   has decreasing norms in $\mathcal H^{\log}_a(\bC_+)$. So
$$
\|h\|_{\mathcal H^{\log}_a(\bC_+)}= \sup_{y\leq 1}\|h(\cdot+iy)\|_{ L^{\log}}.$$
Now we use the fact that the product of a function in $L^1(\R)$ with a function in $\BMO(\R)$ is bounded (see \cite{BGK}) to conclude for the a priori estimate \eqref{apriori}.

Multiplication by $h$ extends to the whole space $\mathcal H^1_a(\bC_+)$ by continuity. It remains to prove that the extension coincides with the multiplication by $h$. This can be done by a routine argument:  convergence in $\mathcal H^{\log}_a(\bC_+)$ implies uniform convergence on compact sets.
\end{proof}

\begin{proof}[\bf Proof of Theorem \ref{the main theorem}]
We have to prove that for every $h\in \mathcal H^{\log}_a(\bC_+)$, there exist $f\in \mathcal H^1_a(\bC_+)$ and $g\in BMOA(\bC_+)$ for which $h= fg$. The proof is very similar to the one of \cite{BIJZ}. Let $h_0$ be the boundary value function of $h$. By the Coifman-Rochberg theorem \cite{CR}, we have
$$b:= \log(e+ |\cdot|) + \log(e + M(|h_0|^{1/2}))\in BMO(\mathbb R).$$

Let $H$ be the Hilbert transform in $\R$. One knows that it can be defined as a continuous operator on $BMO(\R)$. We define $g$ as the Poisson integral of $b + i Hb$, so that $g$ belongs to $BMOA(\bC_+)$ and has $b + i Hb$ as boundary value function. We claim that
$$f= h/g\in \mathcal H^1_a(\bC_+).$$
Indeed, since $b\geq 1$ and $h\in \mathcal H^{\log}_a(\bC_+)$, we obtain that $f\in \mathcal H^{\log}_a(\bC_+)$. Moreover, $f$ has
the boundary value function $f_0={h_0}(b + i Hb)^{-1}$. We write
\begin{eqnarray*}
f_0(x) &\leq& \frac{h^*(x)}{\log(e+|x|) + \log(e + M(|h_0|^{1/2})(x))}\\
&\leq& C\frac{h^*(x)}{\log(e+|x|) + \log(e + (h^*)^{1/2}(x))},
\end{eqnarray*}
where we have used the inequality (\ref{a maximal estimate 3/6}). We use Theorem \ref{maximal characterization} to conclude that $f_0$ is in $L^1(\R)$, then  Theorem \ref{a theorem for boundary value for Hardy spaces}  to conclude that $f\in \mathcal H^1_a(\bC_+)$. This ends the proof.

\end{proof}

\section{Hankel operators and conclusion}
Let us now give a sketch of the proof of Theorem \ref{hankel}. It is equivalent to prove that the Hankel form defined by
$$H_b(f, g):=\langle b, fg\rangle$$
is bounded on  $\mathcal H^1_a(\bC_+) \times \BMOA(\bC_+)$ if and only if $b$ belongs to $\BMOA^{\log}(\bC_+)$. This is a straightforward consequence of the main theorem once one knows that $\BMOA^{\log}(\bC_+)$ is the dual of the space $\mathcal H_a^{\log}(\bC_+)$. The duality has been proven in \cite{Ky} for the real Hardy space $\H^{\log}(\R)$, as well as the continuity of the Hilbert transform, which implies the required duality result. We refer to \cite{BGK, Ky1, Ky} for the definitions and their applications in studying of commutators of singular integral operators.

\bigskip

The main theorem implies also that, on the real line, the embedding of products of functions in $\mathcal H^{1}(\R)$ and $\BMO(\R)$ in $L^1(\R)+\mathcal H^{\log}(\R)$ is sharp: any real function which can be written as the sum of an integrable function and a function in
$\mathcal H^{\log}(\R)$ can also be written as a sum $f_1 g_1+f_2 g_2$, with $f_1$ and $f_2$ in $\mathcal H^{1}(\R)$, $g_1$ and $g_2$ in $\BMO(\R)$. The proof is the same as for the unit disc in \cite{BIJZ}.

\bigskip

One may ask whether results can be generalized to the Siegel domain that is holomorphically equivalent to the unit ball. This is the case for Proposition \ref{fill the gap}. But the converse, with the construction of a function in $\BMOA(\bC_+)$ from its real part, cannot be generalized in higher dimension.

\end{document}